\documentclass[11pt]{article}
\usepackage[utf8]{inputenc}  
\usepackage{amsmath,amssymb,amsfonts,amsthm,bbm}
\usepackage[british]{babel}
\usepackage{tikz,animate,media9}						%
\usepackage{enumerate,enumitem,mdframed}

\usepackage{nicefrac}
\usepackage{crossreftools,booktabs}

\makeatletter
\newcommand{\optionaldesc}[2]{%
  \phantomsection
  #1\protected@edef\@currentlabel{#1}\label{#2}%
}
\makeatother

\usepackage[mono=false]{libertine}
\usepackage[T1]{fontenc}
\usepackage[cmintegrals,libertine]{newtxmath}
\usepackage[cal=euler, calscaled=.95, frak=euler]{mathalfa}
\useosf
\usepackage[a4paper,vmargin={2cm,2cm},hmargin={2.25cm,2.25cm}]{geometry}
\usepackage[font=sf, labelfont={sf,bf}, margin=1cm]{caption}
\usepackage{comment}

\definecolor{NBrown}{HTML}{66220C}
\definecolor{NAqua}{HTML}{00698C}
\definecolor{ForestGreen}{HTML}{228b22}

\usepackage{hyperref}
\usepackage[capitalise]{cleveref}

\parskip \medskipamount

\usepackage{mleftright}
\mleftright

\newtheorem{theorem}{Theorem}[section]
\newtheorem{lemma}[theorem]{Lemma}
\newtheorem{question}[theorem]{Question}
\newtheorem{conjecture}[theorem]{Conjecture}

\newtheorem{corollary}[theorem]{Corollary}

\newtheorem{claim}[theorem]{Claim}

\newtheorem{definition}[theorem]{Definition}

\newcommand{\pr}[1]{\mathbb{P}\!\left(#1\right)}

\newcommand{\prcond}[3]{\mathbb{P}_{#3}\!\left(#1\;\middle\vert\;#2\right)}

\newcommand{\edgebdy}{\partial_E}   %

\DeclareMathOperator{\WUSF}{WUSF}

\renewcommand{\epsilon}{\varepsilon}

\newcommand{\ZZ}{\mathbb{Z}}

\newcommand{\cF}{\mathcal{F}}

\newcommand{\eps}{\varepsilon}

\newcommand{\perco}{\varphi}
\newcommand{\percok}{\perco_k}

\newcommand{\Reff}{R_{\mathrm{eff}}}

\usepackage{tikz}

\usepackage{graphicx}
\usepackage{caption,subcaption}%

\makeatletter
\renewcommand\@dotsep{10000}
\makeatother
\begin{document}

\title{The union of independent USFs on $\ZZ^d$ is transient}

\author{
Eleanor Archer
\thanks{Universit\'e Paris Nanterre, France. \textsf{earcher@parisnanterre.fr}
} \qquad  
Asaf Nachmias 
\thanks{Tel Aviv University, Israel. \textsf{asafnach@tauex.tau.ac.il}
} 
\qquad  
Matan Shalev
\thanks{Tel Aviv University, Israel. \textsf{matanshalev@mail.tau.ac.il}
}
 \qquad  
Pengfei Tang
\thanks{Tianjin University, China. \textsf{pengfei\_tang@tju.edu.cn}
} 
}

\date{\today}

	\maketitle
	\begin{abstract} 
      We show that the union of two or more independent uniform spanning forests (USF) on $\ZZ^d$ with $d\geq 3$ almost surely forms a connected transient graph. In fact, this also holds when taking the union of a deterministic everywhere percolating set and an independent $\eps$-Bernoulli percolation on a single USF sample. %
	\end{abstract}

\section{Introduction}
Given a finite connected graph $G$, the \textit{uniform spanning tree} (UST) on $G$ is a tree drawn uniformly at random from the finite set of spanning trees of $G$. The \textit{wired uniform spanning forest} (WUSF) and \textit{free uniform spanning forest} (FUSF) on an infinite, connected and locally finite graph $G$ are the weak limits of USTs on an exhaustion of $G$ with wired and free boundary conditions respectively. It was shown by Pemantle \cite{Pemantle1991} that these do not depend on the choice of the exhaustion. Moreover, the two measures WUSF and FUSF coincide on $\ZZ^d$. We refer to \cite[Chapter 4 and 10]{Lyons_Peres2016book} for backgrounds on USTs and USFs.

A fundamental result of Morris \cite{Morris2003} states that every component of the WUSF is almost surely recurrent on any graph. A natural question is whether this and other remarkable properties of USFs are stable under various perturbations. In this paper we will be concerned with taking unions of independent WUSFs. We say that a subgraph $\Lambda \subseteq G$ is \textbf{everywhere percolating} if every $x\in G$ is contained in an infinite connected component of $\Lambda$. Also, for any $\eps>0$ we write $\eps$-$\WUSF$ for $\eps$-Bernoulli percolation on the $\WUSF$, that is, conditioned on the $\WUSF$, independently retain each of its edges with probability $\eps$ or erase it otherwise. 

\begin{theorem}\label{thm: main result}
  Let $d\geq 3$ and let $\Lambda$ be an everywhere percolating subgraph of $\ZZ^d$. Then,	for any $\eps>0$ the union of $\Lambda$ and $\eps$-$\WUSF$ on $\ZZ^d$ is almost surely connected and transient. %
\end{theorem}
Benjamini and Tassion \cite{BenTass} proved that this result holds when replacing the $\eps$-$\WUSF$ above with $\eps$-Bernoulli percolation. Since the $\eps$-Bernoulli percolation can be trivially coupled to contain an $\eps$-$\WUSF$, \cref{thm: main result} is strictly stronger. The proof is a ``spatial'' version of the argument in \cite{BenTass}; see \cref{sec:main_proof}. 

The above theorem applies when $\Lambda$ itself is an independent WUSF and $\eps=1$, giving the following answer to a question posed to us by Peleg Michaeli. 
\begin{corollary} \label{cor:main}
  The union of two independent samples of the $\WUSF$ on $\ZZ^d$ with $d\geq 3$ is almost surely connected and transient.
\end{corollary}

Unions of independent spanning trees have origins in the computer science literature, where the union of $k$ independent uniform spanning trees is called a \textit{$k$-splicer} (see \cite{FGRV2014}). These were the first efficient constructions of \textit{sparsifiers} -- sparse graphs of an ambient graph which well approximate its spectrum. It is therefore of interest to understand what other properties are shared by the ambient graph and unions of WUSFs. 
Since the USF is a \emph{critical} statistical physics model, the union of independent $\WUSF$ and $\eps$-$\WUSF$ should be \emph{supercritical} and hence ``inherit'' the structure of the ambient graph. \cref{thm: main result} shows this is the case in $\ZZ^d$. 

It is relatively easy to show that when $G$ is a transitive unimodular nonamenable graph (see \cite{Lyons_Peres2016book} for definitions) almost surely each component of the union of the $\WUSF$ with an independent $\eps$-$\WUSF$ is transient. Indeed, one can readily repeat the proof of \cite[Theorem 13.7]{BLPS2001} to show that when $\eps>0$ is small enough, this union, denoted here by $W$, has infinitely many components, and each has infinitely many ends. Therefore by Lemma 8.35 in \cite{Lyons_Peres2016book} there is an invariant random subforest $\cF \subset W$ such that almost surely each infinite cluster $K$ of $W$ contains a tree of $\cF$ with infinitely many ends. By Corollary 8.20 in \cite{Lyons_Peres2016book}, any such tree has $p_c<1$ and in particular it is transient by Theorem 3.5 in \cite{Lyons_Peres2016book}. By Rayleigh's monotonicity principle, each connected component of $W$ is also transient. We conjecture that this behavior holds in general. %

\begin{conjecture}\label{main_conjecture}
On any transitive transient graph and for any $\eps>0$, every connected component of the union of the $\WUSF$ and an independent $\eps$-$\WUSF$ is almost surely transient.
\end{conjecture}
 
It is not hard to directly argue that on any bounded degree graph and any $p<1$ there exists an integer $k\geq 1$ such that the union of $k$ independent $\WUSF$s dominates $p$-Bernoulli percolation on the graph (one can also use \cite{LSS1997}). Hence by \cite{Hutch23} we deduce that for any transitive transient graph there exists an integer $k$ such that almost surely each component of the union of $k$ independent $\WUSF$s is transient and hence \cref{main_conjecture} asks whether the same holds for ``$k=1+\eps$''. It is the analogue of \cite[Conjecture 1.7]{BLS99} stating that in the same setup of \cref{main_conjecture}, for any $p>p_c$ all infinite $p$-Bernoulli percolation clusters are transient almost surely.

We conclude this section with two open questions related to \cref{thm: main result} that are again based on the intuition that such unions should be ``supercritical''. In the case of supercritical Bernoulli percolation on $\ZZ^d$ it is shown in \cite{Barlow04} that the spectral dimension of the infinite cluster equals $d$ almost surely, and in \cite{Inv1,Inv2,Inv3} it is shown that almost surely the cluster is such that the simple random walk on it diffuses to Brownian motion. It is thus natural to ask the following.

\begin{question}
Consider the union of independent $\WUSF$ and $\eps$-$\WUSF$ on $\ZZ^d$. \begin{enumerate} \item Is the spectral dimension of the union equal to $d$ almost surely? \item Does the simple random walk diffuse to $d$-dimensional Brownian motion almost surely?  
\end{enumerate}
\end{question}

\subsection{Box percolation} In the proof of \cref{thm: main result} we will consider a slightly more general structure, for which we need the notion of\textit{ box percolation}.

\begin{definition}[$(k, \epsilon)$-box percolation]\label{def: keps percolation}
Let $d\geq 1$. Given $k \geq 1$ and $\epsilon \in (0,1)$ we sample a percolation configuration $\percok(\epsilon)$ as follows. Firstly, for each $k \geq 1$ we write
\[
B_k = [-k,k]^d
\]
and for each $z = (z_1, \ldots, z_d) \in (2k\ZZ)^d$ write
\[
B_k^z = [-k,k]^d + z,
\]
and let $E(B_k^z)$ denote its induced edge set. We also let $Q_k^z$ denote the subgraph of $B_k^z$ with the edgeset
\[
\{(x,y) \subset E(B_k^z): \not\exists \ \ell\in \{1,\ldots,d\} \text{ such that }x_{\ell} = y_{\ell}=z_\ell-k\}.
\]
Then, independently for every $z\in (2k\ZZ)^d$, choose a uniform edge from $Q_k^z$ and declare it open with probability $\eps$. Let $\percok(\epsilon)$ be the set of open edges.
\end{definition}

\begin{figure}[ht]
\begin{minipage}[b]{0.5\textwidth}
\centering
\begin{tikzpicture}[scale=0.8]
  \foreach \x in {0,1,...,4}
    \foreach \y in {0,1,...,4}
      \fill (\x,\y) circle (2pt);
  
  \foreach \x in {0,1,...,3}
    \foreach \y in {1,...,4}
      \draw (\x,\y) -- (\x+1,\y);
  
  \foreach \x in {0,1,...,3}
  	\draw[dotted] (\x,0) -- (\x+1,0);
  
  \foreach \x in {1,...,4}
    \foreach \y in {0,1,...,3}
      \draw (\x,\y) -- (\x,\y+1);
      
   \foreach \y in {0,1,...,3}
      \draw[dotted] (0,\y) -- (0,\y+1);

    \node at (2,2) [xshift=5pt, yshift=5pt] {$z$};
   
\end{tikzpicture}
\caption{$Q_2^z$ in $\ZZ^2$, dotted edges are not in $Q_2^z$.} 	
\end{minipage}
\begin{minipage}[b]{0.5\textwidth}	
\centering
	\begin{tikzpicture}[x={(1cm,0)},y={(0,1cm)},z={(0.6cm,0.2cm)}, scale = 2]

  \foreach \x in {0,1}
    \foreach \y in {1,2}
      \foreach \z in {1,2}
        \draw (\x,\y,\z) -- (\x+1,\y,\z);

  \foreach \x in {1,2}
    \foreach \y in {0,1}
      \foreach \z in {1,2}
        \draw (\x,\y,\z) -- (\x,\y+1,\z);

 \foreach \x in {1,2}
    \foreach \y in {1,2}
      \foreach \z in {0,1}
        \draw (\x,\y,\z) -- (\x,\y,\z+1);

\foreach \y in {0,1}
	\foreach \z in {0,1,2}
		\draw[dotted] (0,\y,\z) -- (0,\y+1,\z);

\foreach \y in {0,1,2}
	\foreach \z in {0,1}
		\draw[dotted] (0,\y,\z) -- (0,\y,\z+1);

\foreach \x in {0,1}
	\foreach \y in {0,1,2}
		\draw[dotted] (\x,\y,0) -- (\x+1,\y,0);

\foreach \x in {0,1,2}
	\foreach \y in {0,1}
		\draw[dotted] (\x,\y,0) -- (\x,\y+1,0);
	
\foreach \x in {0,1,2}
	\foreach \z in {0,1}
		\draw[dotted] (\x,0,\z) -- (\x,0,\z+1);

\foreach \x in {0,1}
	\foreach \z in {0,1,2}
		\draw[dotted] (\x,0,\z) -- (\x+1,0,\z);

  \foreach \x in {0,1,...,2}
    \foreach \y in {0,1,...,2}
      \foreach \z in {0,1,...,2}
        \fill (\x,\y,\z) circle (0.5pt);

    \node at (1,1,1) [xshift=5pt, yshift=6pt] {$z$};

\end{tikzpicture}

\caption{$Q_1^z$ in $\ZZ^3$, dotted edges are not in $Q_1^z$.}
\end{minipage}	

\end{figure}

The rather odd choice of edgeset in \cref{def: keps percolation} is made to guarantee two properties: (i) the sets $\{ E(Q_k^z) \}_{z \in (2k\ZZ)^d}$ form a partition of the edges of $\ZZ^d$ and (ii) for any cycle in $\ZZ^d$ there exists $z\in (2k\ZZ)^d$ such that at least two edges of the cycle belong to $E(Q_k^z)$. This is proved in \cref{cl: two edges cycle box} and is used in \cref{lem: domination} to show that for all $\epsilon>0$ and all $k \geq 1$ the $\eps$-$\WUSF$ stochastically dominates $\percok(\epsilon/2d)$. As a result, to prove \cref{thm: main result} it will in fact be sufficient to prove the following theorem.

\begin{theorem}\label{thm: general perco}
Let $d\geq 3$. Take any $k \geq 1$ and any $\epsilon \in (0,1]$. Let $\Lambda$ be an everywhere percolating subgraph of $\ZZ^d$. Then, almost surely, $\Lambda \cup \percok(\eps)$ is connected and transient.
\end{theorem}

\subsection{Organization} 
We first show in \cref{sctn:rule for domination} how \cref{thm: main result} can be formally deduced from \cref{thm: general perco}.  In \cref{sec:main_proof} we prove \cref{thm: general perco}.

\subsection*{Acknowledgements} We would like to thank Peleg Michaeli for posing the question of \cref{cor:main}. We also thank him, Matan Harel and Ofir Karin for useful discussions. This research is supported by the ERC consolidator grant 101001124 (UniversalMap) as well as ISF grants 1294/19 and 898/23.

\section{Proof of Theorem \ref{thm: main result} given Theorem \ref{thm: general perco}}\label{sctn:rule for domination}

For the rest of this section we denote by $\cF$ a sample of the $\WUSF$ and by $\cF_\eps$ a sample of the $\eps$-$\WUSF$. Also recall the definition of $\percok(\eps)$ in \cref{def: keps percolation}. To show that \cref{thm: main result} follows from \cref{thm: general perco}, we will show that $\cF$ stochastically dominates $\percok(1/2d)$. This will follow by the two claims below.

\begin{claim}\label{cl: n edges ust}
Let $H$ be a (possibly infinite) set of edges of $\ZZ^d$ that does not contain a cycle. Then, there is an ordering $\{e_n\}_{n\geq 1}$ of $H$ such that for all $n\geq 1$ and any partition of $\{e_1,\ldots, e_{n-1}\}$ into two sets $A$ and $B$ we have that
\begin{equation*}
	\prcond{e_n \in \cF}{A \subseteq \cF, B\cap \cF = \emptyset}{} \geq \frac{1}{2d}.
\end{equation*}
\end{claim}

\begin{proof}

Note that since the edges $H$ do not contain a cycle, the edge-set $H$ forms a forest. Hence we can order $H$ by first ordering the edges in every tree $T$ of $H$ in such a way that at each step we discover a new vertex, and then concatenate the orderings for all the trees arbitrarily. Write $\{e_n\}_{n\geq 1}$ for this ordering of $H$.
Then for every $n\geq 1$ and $A$ and $B$ partitioning $\{e_i\}_{i=1}^{n-1}$, the edge $e_n$ has at least one endpoint which is not an endpoint of any of the edges in $A\cup B$. Denote this endpoint by $v_n$ and the other one by $u_n$.

Then, by Kirchhoff's formula \cite[Section 4.2 and Section 10.2]{Lyons_Peres2016book} and the spatial Markov property \cite[Proposition 4.2]{BLPS2001}, we have for all $1 \leq k \leq n$ that
\[
\prcond{ e_{n} \in \cF}{A \subset \cF, B\cap \cF = \emptyset}{} = \Reff^{(\ZZ^d / A) \setminus B} (e_{n}),
\]
where the right hand side denotes the wired effective resistance between the endpoints of $e_n$ in the network $(\ZZ^d / A) \setminus B$ which is obtained from a copy of $\ZZ^d$ in which each of the edges in $A$ have been contracted and each of the edges in $B$ have been removed. Note that $\deg (v_n)=2d$ in this graph so by the Nash-Williams inequality \cite[(2.13)]{Lyons_Peres2016book} this resistance is at least $\frac{1}{2d}$.
\end{proof}

\begin{claim}\label{cl: two edges cycle box} Fix $k \geq 1$. For any edge $e\in\ZZ^d$ there exists a unique $z\in(2k\ZZ)^d$, which we denote $g(e)$, such that $e$ is contained in $Q_k^z$. Furthermore, any cycle $(e_1, \ldots, e_n)$ in $\ZZ^d$ contains at least two edges 
$e_{i}, e_{j}$ with $g(e_i) = g(e_j)$. 
\end{claim}
\begin{proof}
For any vertex $u \in \ZZ^d$ write $r_u$ for the number of coordinate of $u$ that are of the form $k+(2k)\ZZ$. Let $e=(u,v)$ be an edge and assume without loss of generality that $r_u \leq r_v$. For any coordinate $\ell \in \{1,\ldots, d\}$ there exists a unique $n_\ell \in \ZZ$ such that 
$u_\ell - 2k n_\ell \in (-k,k]$, we claim that $z=(2kn_1, \ldots, 2kn_\ell)$ is the unique desired $z$. Indeed, for any coordinate $\ell$, if $|z_\ell - u_\ell|<k$, then $|z_\ell-v_\ell|\leq k$ since $u$ and $v$ are neighbors. If $|z_\ell-u_\ell|=k$ (i.e., if $u_\ell = z_\ell + k$) we must have that $u_\ell=v_\ell$ since $r_u \leq r_v$. Hence $(u,v) \subseteq B_k^z$. By our choice of $z$ we have that $u_\ell > z_\ell - k$ for all $\ell$ and therefore $(u,v) \in Q_k^z$. Uniqueness of $z$ follows immediately since the sets $\{Q_k^z\}_{z\in (2k\ZZ)^d}$ are disjoint. 

For the second assertion write $v_0, \ldots, v_{n-1}, v_n = v_0$ for the vertices of the cycle, so that $e_i = (v_{i-1}, v_i)$ (we consider the coordinates mod $n$). For every $j\in \{0,\ldots, n\}$, write $r(j)$ for $r_{v_j}$. %
Note that $|r(j) - r(j+1)| \leq 1$ for all $j$ and that $r(0) = r(n)$. Let $i$ be a global weak minimum of $r:\{0,\ldots,n\}\to \{0,1,\ldots,\}$, so that $r(i) \leq r(i+1) \wedge r(i-1)$. From our description of $z$ in the previous paragraph we obtain that $g(e_{i}) = g(e_{i+1})$. 
\end{proof}

We now combine the previous two claims to obtain the desired stochastic domination.

\begin{lemma}\label{lem: domination}
Let $d\geq 1$ and $k\geq 1$. Then for any $\eps>0$ the forest $\cF_\eps$ stochastically dominates $\percok(\frac{\eps}{2d})$. %
\end{lemma}
\begin{proof}
It suffices to prove for $\eps=1$; the general case follows by applying the same $\eps$-Bernoulli percolation to $\cF$ and $\percok(\frac{1}{2d})$. To sample $\percok(\frac{1}{2d})$, for each $z\in(2k\ZZ)^d$ we sample an independent uniform edge from $Q_k^z$ and denote this collection of edges by $H$. Note that $H$ does not contain any cycles by \cref{cl: two edges cycle box}. Therefore, by revealing the edges of $\cF \cap H$ edge by edge according to the ordering obtained from \cref{cl: n edges ust} we deduce that $\cF \cap H$ dominates $\frac{1}{2d}$-Bernoulli percolation on $H$, and the latter is precisely $\percok(\frac{1}{2d})$. 
\end{proof}%

\section{Proof of Theorem \ref{thm: general perco}} \label{sec:main_proof}
We start with some notation. As in \cref{def: keps percolation}, for $n \geq 1$ and $z \in \ZZ^d$ we let $B_n$ denote the box $[-n,n]^d$ and let $B_n^z = B_n + z$. In addition, when $m \leq n$ we let $A_{m,n}$ denote the annulus $B_n \setminus B_m$. For a subset of vertices $K\subset \ZZ^d$, the \textit{edge boundary} or simply the \textit{boundary} $\edgebdy K$ of $K$ is defined to be the set of edges that connect $K$ to its complement $\ZZ^d \backslash K$. Finally, for $z \in \ZZ^d$ and $k \geq 1$ we take $Q_k^z$ as in \cref{def: keps percolation}.

As explained in the introduction, the proof of \cref{thm: general perco} follows an existing proof of Benjamini-Tassion \cite{BenTass} for the case when $\percok(\eps)$ is replaced by an $\epsilon$-Bernoulli percolation. By standard renormalization techniques, it is in fact sufficient to prove that there exist constants $c>0$ and $C<\infty$ such that for all $\Lambda$ and all sufficiently large $n$ (cf \cite[Lemma 1.1]{BenTass})
\begin{equation}\label{eqn:connection very likely}
\pr{\forall x,y \in B_n, x \text{ is connected to } y \text{ in } \left(\Lambda \cup \percok(\eps)\right) \cap B_{2n}} \geq 1-C\exp(-c\sqrt{n}).
\end{equation}

To establish \eqref{eqn:connection very likely} in the case of Bernoulli percolation, Benjamini and Tassion apply a technique known as \textit{sprinkling} to the component graph of the connected clusters of $\Lambda$ inside a large box; more specifically, they consider the effect of adding an $\eps$-percolation to this component graph by instead adding $4d$ independent $\frac{
\eps}{4d}$-percolations. At each step, they shown that there is a high probability that the number of connected components decreases by a factor of $n^{\nicefrac{1}{4}}$. We cannot directly apply the same strategy since we cannot decompose $\percok(\eps)$ into independent copies of $\percok(\frac{\eps}{4d})$; however we can instead use a form of ``spatial sprinkling" by considering $\percok(\eps)$ on a sequence of $4d$ disjoint annuli.

Moreover, their proof strategy relies on the fact that this component graph is already highly connected at the first step, meaning that there are many edges which could possibly merge components when one adds an $\frac{\eps}{4d}$-percolation. Once again, to obtain the probabilistic bounds they use the fact that the percolation on each of these edges are independent. This is no longer the case for box percolation, but this can also be easily overcome using the geometry of $\ZZ^d$: since the components must reach infinity this means that their boundary is large, and hence must contain many edges in distinct $Q_k^z$. %

\subsection{Proof of Theorem \ref{thm: general perco} assuming \eqref{eqn:connection very likely}}

Write $\vec{e_i}$ for the $i$-th unit vector in $\ZZ^d$ and consider the following random field $(X^{\Lambda}_s)_{s\in \ZZ^d}$. For each $s \in \ZZ^d$ let $X^{\Lambda}_{s}=1$ if and only if
for each $y\in\{ns\pm n\vec{e_i},i=1,\ldots,d\}$ there is a finite path in $\Lambda \cup \perco_k(\epsilon)$ connecting $ns$ and $y$ that also lies entirely in the box $B^{ns}_{2n}$. Let $p=p^{\Lambda}(n)=\inf_{s \in \ZZ^d} \pr{X^{\Lambda}_s=1}$. By \eqref{eqn:connection very likely}, we have that $p^{\Lambda}(n) \to 1$ as $n \to \infty$ (in fact uniformly over choices of $\Lambda$) and moreover, provided that $n > 2k$, we have that $X^{\Lambda}_s$ and $X^{\Lambda}_t$ are independent whenever $||s-t||_{\infty}>4$.

Hence by \cite[Corollary 1.4]{LSS1997} we deduce that $(X^{\Lambda}_s)_{s\in\ZZ^d}$ dominates a supercritical Bernoulli site percolation for all sufficiently large $n$. Since the infinite cluster of a supercritical Bernoulli site percolation on $\ZZ^d$ is transient for $d\geq 3$ \cite{GKZ1993}, it follows that there is a transient connected subgraph formed by the open sites (open means that $X^{\Lambda}_s=1$) of the random field $(X^{\Lambda}_s)_{s\in\ZZ^d}$. By the definition of the random field, this subgraph is roughly equivalent to a subgraph of $\Lambda \cup \perco_k(\epsilon)$ (see definition above \cite[Theorem 2.17]{Lyons_Peres2016book}. Since transience is preserved under rough equivalences \cite[Theorem 2.17]{Lyons_Peres2016book}, we deduce that this subgraph of $\Lambda \cup \perco_k(\epsilon)$ is transient and hence by Rayleigh's monotonicity principle the connected graph $\Lambda \cup \perco_k(\epsilon)$ is also transient.%
\qed

\subsection{Proof of \eqref{eqn:connection very likely}}\label{sctn:BT proof}

\begin{lemma}\label{lem: first bt lemma}
	Let $k>1$. There exists $n_0 = n_0(k,d)<\infty$ such that the following holds. Let $n_0 \leq n\leq m \leq m+\sqrt{n} \leq 8dn$ and let $H$ be a graph on the vertex set $B_{m+\sqrt{n}}$ such that every connected component of $H$ intersects both $B_m$ and $\partial B_{m+\sqrt{n}}$ (in particular $H$ has no isolated vertices). Let $\percok(\eps)$ be as described in \cref{def: keps percolation}. Then there exist $C< \infty$ and $c>0$, not depending on $H$, such that with probability at least $1-C\exp(-c\sqrt{n})$ we have that $H\cup(\percok(\eps) \cap A_{m+2k,m+\sqrt{n}})$ is connected. 
\end{lemma}

For the proof we will fix $m$ and $n$ and for every $0\leq \ell \leq 4d-1$, let 
\[
A_\ell = A_{m+\ell\frac{\sqrt{n}}{4d} + 2k, m+(\ell+1)\frac{\sqrt{n}}{4d}}.
\]
Note that these annuli are disjoint and there is also a ``buffer" of side length $2k$ between disjoint annuli,  so the structure $\percok(\epsilon) \cap A_{\ell}$ is independent for distinct $\ell$.

To prove \cref{lem: first bt lemma} we will use the following claim.

	\begin{claim}\label{cl: sprinkling}
Let $H$ be as in \cref{lem: first bt lemma}, write $H_{\ell}$ for the graph with the same vertex set as $B_{m+\sqrt{n}}$, but endowed with the edge set $H \cup \left(\percok(\eps) \cap \left(\cup_{i=1}^{\ell} A_i\right)\right)$, and write $K(H_\ell)$ for the number of connected components of $H_{\ell}$. Then, there exist $C< \infty$ and $c>0$, such that for every $0\leq \ell \leq 4d-1$, given $\left(\percok(\eps) \cap \left(\cup_{i=1}^{\ell} A_i\right)\right)$ we have that
	\[
	K(H_{\ell+1}) \leq \max\{K(H_\ell)/n^{\nicefrac{1}{4}}, 1\}
	\]
 with probability at least $1-C\exp(-c\sqrt{n})$.
	\end{claim}
	\begin{proof} By our construction the edge set of $H_\ell$ is a subset of $H_{\ell+1}$ so every component of $H_{\ell+1}$ will be a union of connected components of $H_\ell$. We may also assume $H_{\ell}$ is disconnected, otherwise $K(H_{\ell+1})=1$. We claim that every non-trivial union $W$ of connected components of $H_\ell$ must have edge boundary satisfying
$$|\edgebdy W| \cap E(A_{\ell+1}) \geq \sqrt{n}/4d - 2k \geq \sqrt{n}/8d \, ,$$
(when $n$ is large enough). Indeed, since every connected component of $H$ intersects both $B_m$ and $\partial B_{m+\sqrt{n}}$ and since $W$ is non-trivial for every level $r$ in the annulus $A_{\ell+1}$ (that is, for every $r\in [m+(\ell+1)\frac{\sqrt{n}}{4d} + 2k, m+(\ell+2)\frac{\sqrt{n}}{4d}]$), there must be at least one edge in $\ZZ^d$ between $W$ and $W^c$ that has a vertex $v\in W$ with $||v||_\infty=r$ . In particular, the size of $W$'s edge boundary intersected with $A_{\ell + 1}$ is at least the radius of $A_{\ell + 1}$.

		We write $C_1, \ldots , C_{K(H_\ell)}$ for the connected components using the edgeset of $H_\ell$ and take any non-trivial union of $N<K(H_{\ell})$ of these connected components $W=C_{i_1} \cup \ldots \cup C_{i_N}$ and choose $\sqrt{n}/8d$ edges in the edge boundary of $W$ that are also in $A_{\ell+1}$, as guaranteed to exist by the previous paragraph. As every box of the form $Q_k^z$ has at most $d(2k)^d$ edges, we can find a subset of at least $\sqrt{n}/(8d^2(2k)^d)$ boundary edges each belonging to distinct $Q_k^z$. It follows that the probability that $W$ is a connected component in $H_{\ell+1}$ is bounded from above by the probability that all these $\sqrt{n}/(8d^2(2k)^d)$ edges are closed (since $W$ is a connected component). This probability is bounded by 
		\begin{equation*}
			\left(1-\frac{\eps}{d(2k)^d}\right)^{\frac{\sqrt{n}}{8d^2(2k)^d}} \leq \exp\left(-\frac{\eps\sqrt{n}}{8d^3(2k)^{2d}}\right).
		\end{equation*}
		If $K(H_\ell) > n^{\nicefrac{1}{4}}$ and $K(H_{\ell+1}) > K(H_\ell) / n^{\nicefrac{1}{4}}$ then there would have to be a connected component $W$ in $H_{\ell+1}$ which is a union of at most $n^{\nicefrac{1}{4}}$ connected components of $H_\ell$. The probability that this latter event can occur is at most
		\begin{equation}\label{eq: bound for small union}
			\sum_{i=1}^{{n^{\nicefrac{1}{4}}}}\binom{K(H_\ell)}{i}\exp\left(-\frac{\eps\sqrt{n}}{8d^3(2k)^{2d}}\right) \leq n^{\nicefrac{1}{4}}(16dn)^{dn^{\nicefrac{1}{4}}}\exp\left(-\frac{\eps\sqrt{n}}{8d^3(2k)^{2d}}\right)  \leq C\exp\left(-c\sqrt{n}\right),
		\end{equation}
		for some $c>0, C<\infty$ not depending on $n$ or $\ell$, where we have crudely bounded $\binom{K(H_\ell)}{i}\leq |K(H_\ell)|^{n^{\nicefrac{1}{4}}}$ and $|K(H_\ell)| \leq (2(m+\sqrt{n}))^d \leq (16dn)^d$. 
		
		If instead $1 < K(H_\ell) \leq n^{\nicefrac{1}{4}}$ it is only necessary to show that $K(H_{\ell+1}) = 1$ with the same high probability. This follows the same calculation. Indeed, otherwise there would be a non-trivial union of at most $K(H_\ell)$ components which is a component of $H_{\ell+1}$. There are at most $\exp(n^{\nicefrac{1}{4}})$ possibilities for such a union, and again each having probability at most $C\exp(-c\sqrt{n})$. 
			\end{proof}
			
\cref{lem: first bt lemma} now follows straightforwardly. %

\begin{proof}[Proof of \cref{lem: first bt lemma}]
		First note that $H$ has at most $O(m^{d-1})$ connected components (since every component runs through a boundary vertex of $B_m$). Moreover, we can assume that this number of components is less than $n^d$ provided we chose $n_0$ large enough. We can therefore  apply \cref{cl: sprinkling} $4d$ times, once for each $\ell$, and apply a union bound to deduce that with probability at least $1-4dC\exp(-c\sqrt{n})$, the number of components is upper bounded by $\frac{n^d}{n^{4d \cdot 0.25}}=1$.
\end{proof}

For the next lemma, we introduce the following notation for $j \leq \ell$ and a subgraph $X \subset \ZZ^d$: we let $U_{j,\ell}(X)$ denote the number of connected components $K$ of $X$ which have a vertex $v\in H$ with $\|v\|_{\infty} \leq \ell$ but all other vertices $u$ have $||u||_\infty \geq j$. Note that $U_{0,\ell}(X)$ is the number of connected component $K$ of $X$ that contain some $v$ with $\|v\|_{\infty} \leq \ell$.

\begin{lemma}\label{lem: add perco to sqrtn}
Consider an everywhere percolating graph $\Lambda$, restrict it to $B_{8dn}$ and denote the resulting graph by $X$. Then there exist $c>0, C<\infty$ and $n_0 < \infty$ such that for all $n_0 \leq n\leq m \leq m+2\sqrt{n} \leq 8dn$,
\[
U_{0,m}(X\cup (\percok(\eps)\cap A_{m+2k,m+2\sqrt{n}})) \leq \max \left\{ U_{m,m+2\sqrt{n}}(X), 1 \right\}
\]
with probability at least $1-C\exp(-c\sqrt{n})$.
\end{lemma}

Note that if $U_{m,m+\sqrt{n}}(X) = 0$ then the result follows from \cref{lem: first bt lemma}; hence we only need to deal with the case $U_{m,m+\sqrt{n}}(X) \geq 1$. To deal with this case we will essentially follow the same proof strategy as in the proof of \cref{lem: first bt lemma}, but we will consider all components that contribute to $U_{m,m+2\sqrt{n}}(X)$ as one. Here, instead of $A_\ell$ we use $A'_\ell$, defined for  $0\leq \ell \leq 4d-1$ as
\[
A_{\ell}' = A_{m+\sqrt{n}+\ell\frac{\sqrt{n}}{4d} + 2k, m+\sqrt{n}+(\ell+1)\frac{\sqrt{n}}{4d}}.
\]
(Note that the difference with $A_{\ell}$ is the shift by $\sqrt{n}$.) We will need the following claim for the proof.
 
\begin{claim}\label{cl:component reduction special case}
There exist $c>0, C<\infty$ and $n_0 = n_0(k,d)<\infty$ such that the following holds. Let $n_0 \leq n\leq m \leq m+\sqrt{n} \leq 8dn$, take $X$ as in \cref{lem: add perco to sqrtn} and suppose that $U_{m,m+\sqrt{n}}(X) \geq 1$. Let $H$ be the restriction of $X$ to $B_{m+2\sqrt{n}}$.

Define $H'$ to be the graph obtained from $H$ by taking all components not intersecting $B_m$ and contracting them into a single vertex that we call the special component. Write $K(H'_\ell)$ for the number of connected components with the edge set $H'_\ell := H' \cup \left(\percok(\eps) \cap \left(\cup_{i=1}^{\ell} A'_i\right)\right)$. Then, for every $0\leq \ell \leq 4d-1$, we have that
	\[
	K(H'_{\ell+1}) \leq \max\{K(H'_\ell)/n^{\nicefrac{1}{4}}, 1\}
	\]
with probability at least $1-C\exp(-c\sqrt{n})$.
\end{claim}
\begin{proof}
This follows by essentially the same proof as \cref{cl: sprinkling}. Consider some $W$ which is a non-trivial union of connected components of $H_\ell'$. If $W$ consists only of components that traverse $A_{m,m+2\sqrt{n}}$, then it must have an edge-boundary of size at least $\sqrt{n}/8d$ in $A'_{\ell+1}$. If it contains the special component, then this special component must also have such an edge-boundary in $A'_{\ell+1}$, as the special component contains an original component that traversed $A_{m+\sqrt{n}, m+2\sqrt{n}}$ since $U_{m,m+\sqrt{n}}(X) \geq 1$. We can therefore apply the same calculation as in \eqref{eq: bound for small union} to deduce the result.
\end{proof}

\begin{proof}[Proof of \cref{lem: add perco to sqrtn}]
We distinguish between two cases, starting with the case $U_{m,m+\sqrt{n}}(X) = 0$. 
This means that all connected components in $B_{m+\sqrt{n}}$ have a vertex in $B_m$ and hence the result follows from \cref{lem: first bt lemma}.

Otherwise, $U_{m,m+\sqrt{n}}(X) \geq 1$. Take $H$ and $H'$ as in the statement of \cref{cl:component reduction special case}. Similarly to the proof of \cref{lem: first bt lemma}, we can apply \cref{cl:component reduction special case} $4d$ times to obtain that in $H' \cup \left(\percok(\eps) \cap \left(\cup_{i=1}^{4d-1} A'_i\right)\right)$, we have at most one connected component with at least the desired probability. The reason that $U_{0,m}(X\cup (\percok(\eps)\cap A_{m+2k,m+2\sqrt{n}}))$ may possibly be a lot larger is because when expanding the special component, we might further disconnect the graph. However this can create at most $U_{m,m+2\sqrt{n}}(H)-1$ new components, as required.
\end{proof}

The proofs of the next two lemmas are \textbf{exactly} the same as \cite[Lemma 3.2]{BenTass} and \cite[Lemma 1.1]{BenTass}. We include their short proofs here for completeness.

\begin{lemma}\label{lemma: smaller than sqrtn}
	There exist $c>0, C<\infty$ and $n_0 = n_0(k,d)<\infty$ such that the following holds. Take any everywhere percolating graph $\Lambda$. Let $n_0 \leq n \leq  m \leq m+2n \leq 8dn$. Then,
	\begin{equation*}
		U_{0,m}(\Lambda\cup (\percok(\eps)\cap A_{m+2k,m+2n})) \leq \max\{U_{0,m+2n}(\Lambda)/\sqrt{n}, 1\} \, ,
	\end{equation*}
with probability at least $1-C\exp(-c\sqrt{n})$.
\end{lemma}

\begin{proof}
Note that 
	\begin{equation*}
		U_{0,m+2n}(\Lambda) = 	U_{0,m}(\Lambda) + \sum_{\ell=1}^{\sqrt{n}} U_{m+2(\ell-1)\sqrt{n},m+2\ell\sqrt{n}}(\Lambda).
	\end{equation*}
Hence there must be at least one term on the right hand side smaller than $U_{0,m+2n}(\Lambda)/\sqrt{n}$. If it's the first one, then the result is immediate since $U_{0,m}(\Lambda\cup (\percok(\eps)\cap A_{m+2k,m+2n})) \leq U_{0,m}(\Lambda)$. Otherwise we take some $\ell$ such that $U_{m+2(\ell-1)\sqrt{n}, m+2\ell\sqrt{n}}(\Lambda) \leq U_{0,m+2n}(\Lambda)/\sqrt{n}$.	 We then use \cref{lem: add perco to sqrtn} to obtain that
\[
U_{0,m+2(\ell-1)\sqrt{n}}(\Lambda\cup (\percok(\eps) \cap A_{m+2(\ell-1)\sqrt{n} + 2k, m+2\ell\sqrt{n}})) \leq \max \{ U_{m+2(\ell-1)\sqrt{n},m+2\ell\sqrt{n}}(\Lambda),1\} \leq \max \{ U_{0,m+2n}(\Lambda)/\sqrt{n} ,1\} 
\]
with probability at least $1-C\exp(-c\sqrt{n})$. The result then follows since
	\begin{equation*}
		U_{0,m}(\Lambda\cup (\percok(\eps)\cap A_{m+2k,m+2n})) \leq U_{0,m+2(\ell-1)\sqrt{n}}(\Lambda\cup (\percok(\eps) \cap A_{m+2(\ell-1)\sqrt{n} + 2k, m+2\ell\sqrt{n}})). \qedhere
	\end{equation*}
\end{proof}

Now we can give the proof of \eqref{eqn:connection very likely}.

\begin{lemma}
	Let $\eps>0, k \geq 2$. Then, there exist $C<\infty, c>0$ and $N$ such that for every $n\geq N$ and for every everywhere percolating graph $\Lambda$,
	\begin{equation*}
		\pr{\forall x,y\in B_n, x \text{ and } y \text{ are connected inside }B_{2n} \text{ through only } \Lambda\cup\percok(\eps)}{} \geq 1 - C\exp(-c\sqrt{n}).
	\end{equation*}
\end{lemma}
\begin{proof}
We write $n'$ in place of $n$ in the proof to make the use of \cref{lemma: smaller than sqrtn} smoother. We assume that $n'$ is large enough that the results of all the previous lemmas apply to $n=\frac{n'}{4d}$, and look at the restriction of $\Lambda$ to $B_{2n'}$. 	
Similarly to \cite{BenTass}, we consider the disjoint annuli $C_\ell = A_{n'+(\ell-1)(n'/2d) + 2k,n'+\ell n'/2d}$. Successively applying \cref{lemma: smaller than sqrtn} (with $m = 2n' - \ell	n'/2d$ and $n = \frac{n'}{4d}$) for $\ell = 2d,2d-1,\ldots,1$, we obtain that 
\begin{equation*}
	U_{0, 2n' - \ell	n'/2d}(\Lambda \cup \left(\percok(\eps)\cap (\cup_{i=2d-\ell}^{2d}C_i)\right)) \leq \max\left\{\frac{U_{0,2n'-(\ell-1)n'/2d}(\Lambda \cup \left(\percok(\eps)\cap (\cup_{i=2d-\ell+1}^{2d}C_i)\right)}{\sqrt{n'}}, 1\right\}
\end{equation*}
with probability at least $1-2dC\exp(-c\sqrt{n'})$, using the fact that the $C_i$ are separated by buffers of length $2k$.  %
However, provided that $n'$ is large enough we have that $U_{0,2n'}(\Lambda) \leq (n')^d$, and hence after applying the lemma $2d$ times we obtain that $U_{0,2n'}\left(\Lambda \cup \left(\percok(\eps) \cap \left(\cup_{i=1}^{2d}C_i\right)\right)\right) = 1$ with the desired probability.
This means that every vertex in $B_{n'}$ is in the same connected component of $\Lambda\cup\left(\percok(\eps) \cap A(n',2n')\right)$ inside $B_{2n'}$, as required.
\end{proof}

\bibliographystyle{abbrv}

\end{document}